\documentclass[a4paper,11pt]{article}

\usepackage{graphicx}
\usepackage[affil-it]{authblk}

\usepackage{alltt} 

\usepackage{amsfonts}
\usepackage{amsmath}
\usepackage{amssymb}
\usepackage{amsthm}

\newcommand{\ZZ}{\mathbb Z}

\def\P{{\mathcal P}}

\newtheorem{theorem}{Theorem}[section]
\newtheorem{proposition}[theorem]{Proposition}
\newtheorem{lemma}[theorem]{Lemma}
\newtheorem{corollary}[theorem]{Corollary}

\theoremstyle{definition}
\newtheorem{definition}[theorem]{Definition}

\theoremstyle{remark}
\newtheorem{example}[theorem]{Example}
\theoremstyle{remark}

\begin{document}

\title{Orbit Equivalence Rigidity of Equicontinuous Systems}

\author{Mar\'{\i}a Isabel Cortez%
  \thanks{The research of the first author was supported by Anillo Research Project 1103 DySyRF and Fondecyt Research Project
1140213.}}
\affil{Departamento de Matem\'atica y Ciencia de la Computaci\'on, Universidad de Santiago de Chile \\  maria.cortez@usach.cl}

\author{Konstantin Medynets%
  \thanks{The second author was supported by NSA grant H98230CCC5334.}}
\affil{Department of Mathematics, United States Naval Academy  \\ medynets@usna.edu}

\date{Dated: \today}
\maketitle

\begin{abstract} The paper is focused on the study of continuous orbit equivalence for minimal equicontinuous systems.   We  establish that every equicontinuous system  is topologically conjugate to a profinite action, where the finite-index subgroups are not necessarily normal.  We then show that two profinite actions $(X,G)$ and $(Y,H)$ are continuously orbit equivalent if and only if the groups $G$ and $H$ are virtually isomorphic and the isomorphism preserves the structure of the finite-index subgroups defining the actions.  As a corollary, we obtain a dynamical classification of the restricted isomorphism between generalized Bunce-Deddens $C^*$-algebras.
We show that for  minimal equicontinuous $\mathbb Z^d$-systems continuous orbit equivalence implies that the systems are virtually piecewise conjugate.   This result extends Boyle's flip-conjugacy theorem.
We also show that the topological full group of a minimal equicontinuous system $(X,G)$  is amenable if and only if the group $G$ is amenable.
\end{abstract}

\section{Introduction}\label{SectionIntroduction}

Let $X$ be a compact topological space and $G$ be a group acting by homeomorphisms on $X$. The pair $(X,G)$ is called a {\it dynamical system}. In the paper all groups are assumed to be countable and topological spaces are homeomorphic to the Cantor set. The action of a group element $g\in G$ on a point $x\in X$ will be denoted by $g\cdot x$.

Let $(X,G)$ and $(Y,H)$ be dynamical systems. We say that $(X,G)$ and $(Y,H)$ are {\it orbit equivalent} if there is a homeomorphism $\psi : X\rightarrow Y$ establishing a bijection between $G$-orbits and $H$-orbits.  We say that the dynamical systems are {\it continuously orbit equivalent} if they are orbit equivalent and for any $x\in X$ and $g\in G$ there is $h\in H$ such that $g = \psi^{-1} \circ h\circ \psi$ on a {\it clopen neighborhood} of $x$.

The study of continuous orbit equivalence of dynamical systems is primarily motivated by its applications to the classification theory of crossed product $C^*$-algebras and geometric group theory. Let us denote by $C_r(X,G)$ the reduced crossed product $C^*$-algebra associated to a dynamical system $(X,G)$. It turns out that the $C^*$-algebras $C_r(X,G)$  and $C_r(Y,H)$ are isomorphic via an isomorphism mapping $C(X)$ onto $C(Y)$ if and only if the systems $(X,G)$ and $(Y,H)$ are continuously orbit equivalent. This result was originally established by T.~Giordano, I.~Putnam, C.~Skau in \cite[Theorem 2.4]{GPS95} for minimal $\mathbb Z$-actions, generalized to topologically free $\mathbb Z$-systems by J.~Tomiyama \cite[Theorem 2]{Tomiyama:1996}, and later to arbitrary topologically free dynamical systems by J.~Renault \cite[Proposition 4.13]{Renault:2008}, see also \cite[Theorem 5.1]{Matui:2012}.

 To any Cantor dynamical system $(X,G)$ we can associate a countable group $[[G]]$ defined as the set of all homeomorphisms of $X$ that locally coincide with elements of $G$. The group $[[G]]$ is termed the {\it topological full group of} $(X,G)$. No topology is assumed on $[[G]]$. The adjective ``topological'' is a historical term used to differentiate it from full groups arising in ergodic theory.  In \cite{GPS99}, the authors proved a topological version of the Dye theorem by showing that two minimal $\mathbb Z$-systems are continuously orbit equivalent if and only if their topological full groups are isomorphic as abstract groups. This result was generalized in \cite{Med11} to arbitrary, with some minor assumptions, group actions.  This reconstruction result implies that the structure of $G$-orbits can be fully recovered from algebraic properties of the associated full groups. We refer the reader to the paper \cite{GrigorchukMedynets} for a discussion of algebraic properties of full groups associated to minimal $\mathbb Z$-systems.

 We notice that topological full groups have recently found applications in geometric group theory -- they were used to construct the first examples of {\it infinite simple finitely generated amenable groups} \cite{JuschenkoMonod:2013}.

 The goal of this paper is to study continuous orbit equivalence for equicontinuous systems. Our focus on equicontinuous systems is motivated by the fact that they are, in a sense, primary building blocks for general dynamical systems as every dynamical system has a maximal equicontinuous factor.

Let $G$ be a residually finite group and $\{G_n\}_{n\geq 0}$ be a nested sequence of finite index subgroups. {\it We do not assume that these subgroups are normal}. The group $G$ has a profinite  action  on the coset tree, the inverse limit, $X = \lim (G/G_n,\pi_n)$, where $\pi_n: G/G_{n}\rightarrow G/G_{n-1}$ is the natural quotient map. We call this profinite action a {\it $G$-odometer}. If the  subgroups $\{G_n\}_{n\geq 1}$ are normal, then we call $(X,G)$ an {\it exact $G$-odometer}.  We borrowed this terminology (though we slightly modified it) from \cite{CP} and \cite{Cor06}. We note that every $G$-odometer is equicontinuous.

  In \cite{Huang:1979} Huang showed that if a group $G$ admits an effective  equicontinuous action on a compact set, then $G$ must necessarily be maximally almost periodic.  Every maximally almost periodic group $G$ is embeddable into its   Bohr compactification \cite[p. 158]{Loomis:Book}, which is a compact group. Using the Peter-Weyl and Mal'cev theorems one can show that every finitely generated maximally almost periodic group is also  residually finite. In Theorem \ref{TheoremEquicontinuousConjugateOdometers} we  will give a direct proof of the fact that if a dynamical system $(X,G)$ is free, minimal, and equicontinuous, then the group $G$, which is not assumed to be finitely generated, is necessarily residually finite. We will also show that $(X,G)$ is topologically conjugate to a $G$-odometer.

In \cite{Boyle:1983}, \cite{BoyleTomiyama:1998}, M.~Boyle and J.~Tomiyama showed that  continuous orbit equivalence of minimal $\mathbb Z$-systems   is equivalent to  the flip-conjugacy of dynamical systems (conjugacy up to time reversal) (Definition \ref{DefinitionConjugacy}). In the literature such results  are often referred to as {\it Rigidity Theorems}. The main goal of the paper is to establish a rigidity theorem  for equicontinuous systems. We show that  the class of continuous orbit equivalence for an equicontinuous system consists of the systems virtually conjugate to it.

The following are the main results of the paper. The proofs are given in  Section \ref{SectionRegidityResults}.   Recall that a dynamical system $(X,G)$ is called {\it free} if $g\cdot x =x$, $x\in X$, $g\in G$, implies that $g= e$, the group identity.

\begin{theorem} Let $(Y,H)$ be a free dynamical system. If  $(Y,H)$ is continuously orbit equivalent to a free odometer $(X,G)$ with $G$ a finitely generated group, then $(Y,H)$ is an odometer and the groups $G$ and $H$ are commensurable.
\end{theorem}

\begin{theorem}\label{TheoremMainIntro} Let $(X,G)$ and $(Y,H)$ be free odometers. Then the following are equivalent:

(1) $(X,G)$ and $(Y,H)$ are continuously orbit equivalent.

(2) The topological full groups $[[G]]$ and $[[H]]$ are isomorphic.

(3) $C_r(X,G)$  and $C_r(Y,H)$  are isomorphic via an isomorphism mapping $C(X)$ onto $C(Y)$.

(4) There exist nested sequences of finite index subgroups  $\{G_n\}_{n\geq 0}$ and $\{H_n\}_{n\geq 0}$ determining the structure of $(X,G)$ and $(Y,H)$ as odometers, respectively, and a group isomorphism $\theta : H_0\rightarrow G_0$ such that $[G:G_0] = [H:H_0]$ and $\theta(H_n) = G_n$ for every $n\geq 0$.

\end{theorem}

We note that the crossed product $C^*$-algebras associated to $G$-odometers  were earlier studied by Orfanos \cite{Orfanos:2010} and were termed   {\it generalized Bunce-Deddens algebras}. These algebras coincide with classical Bunce-Deddens algebras whenever $G=\mathbb Z$ \cite[Section V.3]{Davidson:Book}.

As a corollary of the main results, we obtain that for minimal equicontinuous $\mathbb Z^d$-actions the  continuous orbit equivalence implies  virtual piecewise conjugacy (Theorem \ref{TheoremRigidityZnOdometers}).

 One of the big open problems in the topological orbit equivalence theory is to describe systems whose topological full groups are amenable. We note that topological full groups of minimal $\mathbb Z$-systems are amenable \cite{JuschenkoMonod:2013}, see also \cite{JuschenkoNekrashevychSalle}. However, there are minimal $\mathbb Z^2$-systems with non-amenable full groups \cite{ElekMonod:2013}. In the following theorem we show that  equicontinuous minimal systems have amenable topological full groups whenever the acting group is amenable. This, in particular, implies that (1) the topological full group of  a direct product of $\mathbb Z$-odometers is amenable; (2) topological full groups of Heisenberg odometers are amenable, see \cite{LightwoodSahinUgarcovici:2014} for more details on Heisenberg odometers.

\begin{theorem}[Corollary \ref{CorollaryFullGroupsAmenable}] Let $(X,G)$  be a free minimal equicontinuous system. Then the topological full group $[[G]]$ is amenable if and only if the  group $G$ is amenable.
\end{theorem}

  Historically, the orbit equivalence rigidity phenomena were first discovered in the measurable dynamics, see, for example,  \cite{Ioana:2011} and references therein. In  \cite{Ioana:2011}, Ioana studies the measurable cocycle superrigidity for profinite actions of property $(T)$ groups. We note that, in spirit, his results have some similarities with ours, though their scopes and the  techniques employed  are completely different.  We would also like to mention a recent preprint  \cite{Li:2015}, where the author establishes a number of  rigidity results for various topological dynamical systems.

The structure of the paper is as follows. In section \ref{SectionEquicontinuousSystems} we show that every minimal equicontinuous system is conjugate to a $G$-odometer. The main results of the paper are established in Section \ref{SectionRegidityResults}. Section \ref{SectionTopFullGroups} is devoted to the study of topological full groups associated with odometers.

{\bf Acknowledgement:} This project was started when the second-named author visited  the mathematics department of Universidad de Santiago de Chile. He  would like to thank the department for the hospitality during his visit.


\section{Equicontinuous systems}\label{SectionEquicontinuousSystems}

A dynamical system $(X,G)$ is called {\it effective} if for each $g\in G$, $g\neq e$, there exists $x\in X$ such that $g\cdot x\neq x$.  A dynamical system $(X,G)$ is called {\it minimal} if every $G$-orbit is dense in $X$.  A subset $Y\subset X$ is called a {\it minimal component} if $Y$ is $G$-invariant and $(Y,G)$ is minimal.

\begin{definition}\label{DefinitionConjugacy} (1) Dynamical systems $(X_1,G)$ and $(X_2,G)$ are called {\it conjugate} if there exists a homeomorphism $\psi:X_1 \rightarrow X_2$ such that $\psi(g\cdot x) = g\cdot \psi(x)$ for every $g\in G$ and $x\in X_1$.

(2) Dynamical systems $(X_1,G_1)$ and $(X_2,G_2)$ are called {\it conjugate up to a group isomorphism}, if there exist a group isomorphism  $\theta : G_1\rightarrow G_2$ and a homeomorphism $\psi: X_1\rightarrow X_2$ such that $\psi(g\cdot x) = \theta{(g)}\cdot \psi(x)$ for every $g\in G_1$ and $x\in X_1$. In this case, we will also say that {\it the systems are $\theta$-conjugate}.
\end{definition}

\begin{definition} A dynamical system $(X,G)$ is called {\it equicontinuous} if the collection of maps defined by the action of $G$ is  uniformly equicontinuous, i.e., if for every $\varepsilon>0$ there exists $\delta>0$ such that $d(x,y)\leq \delta$ implies  $d(g\cdot x, g\cdot y) <\varepsilon$, for every $g\in G$.
\end{definition}

Every equicontinuous system is a disjoint union of its minimal components,  see Corollary 10 in Ch.1 and Theorem 2 in Ch.2 in \cite{Au}.
 In  \cite{AGW} the authors showed that every distal action of a finitely generated group on a compact zero-dimensional metric space is equicontinuous.

\begin{example}\label{ExampleModelEquicontinuous} Consider a discrete group $G$. Suppose that there exists a homomorphism  $\varphi: G\to K$ into a compact group $K$ such that $\varphi(G)$ is dense in $K$. Fix  a closed subgroup $H$ of $K$. Define an action of $G$  on the left cosets $K/H$ as follows
$$
g\cdot( kH)=\varphi(g)kH, \mbox{ for every } k\in K \mbox{ and } g\in G.
$$
The system $(K/H,G)$ is minimal and equicontinuous \cite[Page 39]{Au}. According to Theorem 6 in  \cite[Ch.3]{Au},  every minimal equicontinuous system is conjugate to a system of the form $(K/H,G)$.
\end{example}

 Let $G$ be a group and $\{G_n\}_{n\geq 0}$ be a decreasing sequence of finite-index subgroups (not necessarily normal). Let $\pi_n : G/G_{n}\rightarrow G/G_{n-1}$ be the natural quotient map.  Consider the inverse limit $X = \lim_n(G/G_n,\pi_n).$ We remind that $X$ consists of tuples $(g_0,g_1,g_2,\ldots) \in \prod_{n=0}^\infty G/G_n$ such that $\pi_n(g_{n}) = g_{n-1}$ for all $n\geq 1$.  The topology on $X$ is generated by the clopen sets $\{\{g_n\}\in X : g_i = a_i \}$, where $a_i\in G/G_i$.

 The group $G$ acts continuously on $X$ by left multiplication, i.e.,  the action is defined by $g\cdot \{h_i\}:=\{g h_i\}$, where  $g\in G$ and $\{h_i\}\in X$.   We note that the dynamical system $(X,G)$ is minimal and equicontinuous  \cite[Ch.2]{Au}. However,  the action of $G$ on $X$ does not always have to be free.

 If the  subgroups $\{G_i\}$ are normal in $G$, then  $X$ is a profinite group and there is a natural homomorphism $\tau: G \rightarrow X$ that defines the action of $G$   on $X$. In this case,  the action of $G$ is free  if and only if $\bigcap_{n = 0}^\infty G_n = \{e\}$.
  Indeed,  $\tau(g) = e$ iff for every  $g \in G_i$ for all $i\geq 0$. Thus, $\tau$ is an embedding if and only if $\bigcap_{i\geq 0}G_i = \{e\}$.

\begin{definition} Let $G$ be a group and $\{G_n\}$ be a sequence of finite-index subgroups and let $X$ be the inverse limit as above.  We will call the dynamical system $(X,G)$ a $G$-{\it odometer} or, simply, an {\it odometer} when the group $G$ is clear from the context. If the finite-index subgroups $\{G_n\}$   are normal, then we call $(X,G)$ an {\it exact $G$-odometer}.
  \end{definition}

 Every $G$-odometer is a factor of an exact $G$-odometer \cite[Proposition 1]{CP}.
We note that odometers are sometimes referred to as {\it profinite actions}  \cite{AbertAlek:2012}, \cite{Ioana:2011}.

 \begin{definition} Let $(X,G)$ be a dynamical system.

 (1)  For a subset $U\subset X$ and point $x\in X$, the set of {\it return times of the point $x$ to $U$} is defined as  $T_U(x) = \{g\in G | g\cdot x\in U\}$.

 (2) A point $x\in X$ is said to be {\it regularly recurrent}  if for every clopen neighborhood $U$ of $x$ there exists a finite-index subgroup $K\subset G$ such that $K\subset T_U(x)$.
\end{definition}

\begin{proposition}[Theorem 2 in \cite{CP}]\label{PropositionRegularlyRecurrent} Let $(X,G)$ be a minimal system whose every point is regularly recurrent. Then $(X,G)$ is topologically conjugate to a $G$-odometer.
\end{proposition}

 The following result shows that every minimal equicontinuous system on a Cantor set is conjugate to a $G$-odometer. We would like to mention that after the paper was submitted, the authors of \cite{DyerHurderLukina:2015} announced several results similar in spirit to the following theorem.

\begin{theorem}\label{TheoremEquicontinuousConjugateOdometers}
Let $(X,G)$ be a free equicontinuous minimal  system. Then  the group $G$ is residually finite and $(X,G)$ is conjugate to a $G$-odometer.
\end{theorem}
\begin{proof}

 Every minimal equicontinuous system is conjugate to a system $(K/H,G)$ as described in Example \ref{ExampleModelEquicontinuous}. The freeness of $(X,G)$ implies that $\varphi$ is an embedding. Note that the group $G$ acts minimally on the group $K$ by translations \cite{Au}.

It follows from the arguments in \cite[Ch. 3, Thm 6]{Au} that  the group $K$ arises as a closed subset of $X^X$. Thus, $K$ is a zero-dimensional topological group. Let $U$ be a clopen neighborhood of $e$ in $K$. Theorem 7.7 in \cite{HR} implies that  there exists a clopen subgroup $L$ of $K$ contained in $U$. By  minimality of the action of $G$, there exists a finite set $F\subset G$ such that $K=\varphi(F)L$.

 Note that for any $g\in  G$ there exists $l\in L$ and $f\in F$ such that $\varphi(g) = \varphi(f)l$. Hence, $l\in L\cap \varphi(G)$. Setting   $\Lambda=L\cap \varphi(G)$, we see that  $$\varphi(G)=   \varphi(F)\Lambda.$$
It follows that  $$G = F\varphi^{-1}(\Lambda).$$
 Therefore, $G' = \varphi^{-1}(\Lambda)$ is a finite-index subgroup of $G$. Note that for every $g\in G'$, we have that $g\cdot e = g\in \Lambda\subset U$.  Let $\{U_n\}$ be a nested sequence of clopen neighborhoods of $e$ in $K$ such that $\bigcap_{n\geq 1}U_n = \{e\}$. For each $U_n$ we can construct a finite index subgroup $G_n$ such that $G_n\cdot e\subset U_n$. Notice that $\bigcap_{n\geq 1}G_n = \{e\}$. This shows that $G$ is a residually finite group.

 By construction,   the group $G'$ is contained in the set of return times of $e$ to $U$. Since the clopen set $U$ is arbitrary, $e$ is a regularly recurrent point. Therefore, every element of the group $K$ is regularly recurrent. Using the fact that $(K/H,G)$ is a factor of $(K,G)$,  we obtain that every point in $K/H$ is regularly recurrent. Applying Proposition \ref{PropositionRegularlyRecurrent}, we conclude that $(K/H,G)$ is topologically conjugate to a $G$-odometer.
\end{proof}

\begin{corollary}\label{CorollaryDisjointUnionOdometers} (1) Every equicontinuous dynamical system on a Cantor set is topologically conjugate to a disjoint union of odometers.

(2) Every equicontinuous dynamical system is measure-theoretically conjugate to an  odometer.
\end{corollary}
\begin{proof}
(1) Using  Corollary 10 in Ch.1 and Theorem 2 in Ch.2 from \cite{Au}, we see that every equicontinuous system is a disjoint union of its minimal components. Theorem \ref{TheoremEquicontinuousConjugateOdometers} implies that every minimal component is conjugate to a $G$-odometer.

(2) It follows from (1) that that every ergodic measure must be supported by a minimal component, which is conjugate to a $G$-odometer.
\end{proof}

%
%

\section{Rigidity Theorems}\label{SectionRegidityResults}

In this section we establish that two minimal equicontinuous systems are continuously orbit equivalent if and only if they are ``almost virtually'' conjugate. This can be seen as an extension of  Boyle's flip conjugacy theorem \cite{Boyle:1983}, \cite{BoyleTomiyama:1998} to the case of free minimal equicontinuous systems.
In view of Theorem \ref{TheoremEquicontinuousConjugateOdometers}, we can assume that the systems of interest are free odometers and all groups are residually finite.

\begin{definition}\label{DefinitionStructurConj} Let $(X,G)$ and $(Y,H)$ be odometers. We say that $(X,G)$ and $(Y,H)$ are  {\it structurally conjugate} if  there exist decreasing sequences of finite index subgroups   $\{G_n\}_{n\geq 0}$ and $\{H_n\}_{n\geq 0}$ that determine $(X,G)$ and $(Y,H)$, respectively, and an isomorphism $\theta: H_0\rightarrow G_0$ such that  $\theta(H_n) = G_n$, $n\geq 0$, and $[G:G_0] = [H:H_0]<\infty$.
\end{definition}

The main result of this section is the proof of the fact that two odometers are continuously orbit equivalent if and only if they are structurally conjugate.

Consider a $G$-odometer $(X,G)$. Denote by $\mathbf{e}$ the element $\{e_n\}_{n\geq 0}\in X$, where $e_n$ is the coset in $G/G_n$ corresponding to the group $G_n$. Set $C_n = [e]_n$. Note that $C_n = \overline{G_n\cdot \mathbf e}$. Furthermore, the group $G_n$ is precisely the set of {\it return times} to $C_n$. In other words, $G_n = \{g\in G : g(x) \in C_n\}$ for any $x\in C_n$. Note also that the induced system $(C_n,G_n)$ is a $G_n$-odometer determined by the sequence of subgroups $\{G_{i}\}_{i\geq n+1}$.

\begin{lemma}\label{LemmaConjugacyInducedOdometers} Let $(X,G)$ be a dynamical system. Suppose that there exist a finite-index subgroup $G_0\subset G$, a system $\{f_0,\ldots,f_{n-1}\}$ of representatives for $G/G_0$, and a clopen set $C\subset X$  such that

(i) $G_0=\{g\in G : g(x)\in C\}$ for every $x\in C$;

(ii) the system $(C,G_0)$ is a $G_0$-odometer;

(iii) the family $\{f_0\cdot C,\ldots,f_{n-1}\cdot C\}$ is a clopen partition of $X$.

\noindent
Then $(X,G)$ is conjugate to a $G$-odometer.
\end{lemma}
\begin{proof}
Fix a decreasing sequence of finite-index subgroups $\{G_n\}_{n\geq 1}$ of $G_0$ that determine the structure of the $G_0$-odometer $(C,G_0)$. Consider the $G$-odometer $(Y,G)$ corresponding to the sequence of subgroups $\{G_n\}_{n\geq 0}$.  We claim that the systems $(X,G)$  and $(Y,G)$ are conjugate.

Note that $(C,G_0)$ and $([e]_0,G_0)$, $[e]_0\subset Y$, are conjugate. Denote the homeomorphism implementing the conjugacy between the systems by $\varphi$. Extend it to a homeomorphism $\varphi : X \rightarrow Y$ as follows: for  $y\in X$ find unique $x\in C$ and $f_i$, $i=0,\ldots,n-1$, with $f_i\cdot x = y$ and set $\varphi(f_i\cdot x) = f_i\cdot \varphi(x)$.

  Fix  $g\in G$ and $y\in X$. Let $x\in C$  and $f_i$ be as above. Find $h\in G_0$ and $f_j$ such that $gf_i = f_j h$. Then $$\varphi(g\cdot y) = \varphi(gf_i\cdot x) = \varphi(f_j h\cdot x) = f_j \cdot \varphi(h\cdot x) = f_jh\cdot \varphi(x) = g f_i\cdot \varphi(x) = g\cdot \varphi(y).$$
This shows that $\varphi$ is $G$-equivariant, which implies the result.
\end{proof}

Suppose the dynamical systems $(X,G)$ and $(Y,H)$, with  groups acting freely, are  continuously orbit equivalent. Let $\varphi: Y \rightarrow X$ be a homeomorphism implementing the orbit equivalence. Define a function $f : H\times Y \rightarrow G$  by $$f(h,y)\cdot \varphi(y) = \varphi(h\cdot y)\mbox{ for every }h\in H\mbox{ and }y\in Y,$$ dubbed an {\it orbit cocycle}. Note that $f$ satisfies the {\it cocycle identity:} $$f(h_1h_2,y) = f(h_1,h_2\cdot y) f(h_2,y)\mbox{ for every }h\in H\mbox{ and }y\in Y.$$

By construction, the function $f: H\times Y\rightarrow G$ is continuous and for every $y\in Y$  $f(\cdot,y) : H\rightarrow G$ is a bijection. Note that the ``dual'' cocycle $q : G\times X\rightarrow H$ is also continuous.

\begin{theorem}[Rigidity Theorem]\label{TheoremMainRigidity} (1) Let $H$ be a finitely generated residually finite group and $(Y,H)$ be  an $H$-odometer. Suppose  a free dynamical system $(X,G)$ is continuously orbit equivalent to $(Y,H)$.  Then  $(X,G)$ is conjugate to a free odometer that is structurally conjugate to  $(Y,H)$.

(2) Conversely, let $(X,G)$ and $(Y,H)$ be odometers. Suppose $(X,G)$ and $(Y,H)$ are structurally conjugate. Then $(X,G)$ and $(Y,H)$ are continuously orbit equivalent.
\end{theorem}
\begin{proof} (1)  Assume that  $(X,G)$ and $(Y,H)$ are continuously orbit equivalent. By conjugating the system $(Y,H)$, we can assume that both groups $G$ and $H$ act on the same space $X$ and share the same orbits.
Denote by $f:H\times X\rightarrow  G$ the  orbit cocycle defined by $f(h,x)\cdot x = h\cdot x$.  Notice that for  given $h\in H$, $f(h,\cdot): X\rightarrow G$ is a continuous function.

Fix a symmetric set of generators $\{s_1,\cdots,s_r\}$ for $H$. Find a clopen partition $O_1\sqcup O_2\sqcup\ldots \sqcup O_p=X$ such that the cocycle $f(s_i,\cdot)|O_j=const$ for every $i$ and $j$. Let $\delta> 0$
 be Lebesgue's number of the partition $\{O_1,\ldots,O_p\}$.  Since the system $(X,H)$ is equicontinuous, we can find a clopen refinement  $U_1\sqcup\ldots \sqcup U_k = X$ of the partition $\{O_1,\ldots, O_p\}$ such that if $x,y\in U_i$ for some $i$, then $d(h\cdot x, h\cdot y) < \delta$ for every $h\in H$. Here $d$ is a metric  compatible with the topology. Therefore, if $x,y\in U_i$ and $h\cdot x \in O_j$, $h\in H$, then $h\cdot y\in O_j$.
 It follows that if $x,y\in U_i$, then for any $h\in H$ and $s_j$, we have that $f(s_j,h\cdot x) = f(s_j,h\cdot y)$.

Consider an arbitrary element $h=s_{i_1}\cdots s_{i_m}\in H$. If $x,y\in U_i$, then $$f(h,x) = \prod_{l=1}^m f(s_{i_l},s_{i_{l+1}}s_{i_{l+2}}\cdots s_{i_m}\cdot x) =
 \prod_{l=1}^m f(s_{i_l},s_{i_{l+1}}s_{i_{l+2}}\cdots s_{i_m} \cdot y) = f(h,y).$$

Let $\{H_n\}_{n\geq 0}$ be a sequence of  subgroups that determine the structure of the odometer $(Y,H)$. Choose $n>0$ such that the partition $\{C_0,\ldots,C_{q-1}\}$ into cosets $H/H_n$  refines $\{U_1,\cdots,U_k\}$. We will assume that the set $C_0$ corresponds to $[e]_n$.  Set  $H' = H_n$. Note that the set  $C_0$ is $H'$-invariant. Note also that $f(h,x)=f(h,y)$ for every $h\in H$ and $x,y\in C_0$.

 Fix  $x\in C_0$. Set $\theta(h) = f(h,x)$. The definition of $\theta$ is independent of $x$. If $h_1,h_2\in H'$, then
$$\theta(h_1h_2) = f(h_1h_2,x)= f(h_1,h_2\cdot x)f(h_2,x) = f(h_1, x)f(h_2,x) = \theta(h_1)\theta(h_2).$$

Set $G' = \theta(H')$. Note that if $\theta(h) = e$, then, in view of freeness of the action, $h=e$.  It follows that  $\theta : H' \rightarrow G'$ is an isomorphism and that $(C_0,G')$ is a $G'$-odometer.
Fix representatives $\{h_0,\ldots,h_{q-1}\}$ for cosets in $H/H'$. Then
$$\begin{array}{lll} G & = &f(H,x) \\
 & = & f\left(\bigsqcup_{j=0}^{q-1} h_j\cdot H',x\right) \\
 & = & \bigsqcup_{j=0}^{q-1}  f\left( h_j,H'\cdot x\right)f(H',x) \\
 & = & \bigsqcup_{j=0}^{q-1}f\left( h_j,x\right)G'.\end{array}$$ This implies that $[H:H'] = [G:G'] <\infty$. Note that $h_i\cdot C_0 = f(h_i,x) \cdot C_0$. Hence, $\{f(h_0,x)\cdot C_0,\ldots, f(h_{q-1},x)\cdot C_0\}$ is a clopen partition of $X$ and $G'$ is the set of return times to $C_0$.  Applying Lemma \ref{LemmaConjugacyInducedOdometers} we obtain that $(X,G)$ is a $G$-odometer. Setting $G_m = \theta(H_m)$, $m\geq n$, we see that $(X,G)$ and $(Y,H)$ are structurally conjugate.

(2)  Conversely, suppose that $(X,G)$ and $(Y,H)$ are structurally conjugate. Let $\{G_n\}_{n\geq 0}$, $\{H_n\}_{n\geq 0}$ and $\theta : H_0 \rightarrow G_0$ be as Definition \ref{DefinitionStructurConj}.  Let $\{C_0,\ldots,C_{n-1}\}$ be the collection of clopen sets corresponding to the cosets  $G/G_0$, with  $C_0$ corresponding to $G_0$. Note that  $C_0$ is $G_0$-invariant. Similarly, let $\{D_0,\ldots, D_{n-1}\}$ be the collection of clopen sets corresponding to the cosets $H/H_0$ with $D_0$ corresponding to $H_0$.

 Note that $(D_0,H_0)$  and $(C_0,G_0)$ are $\theta$-conjugate. Denote by $\varphi: D_0\rightarrow C_0$ the homeomorphism implementing the conjugacy, i.e., $$\varphi(h\cdot y) = \theta(h)\cdot \varphi(y)\mbox{ for all }h\in H_0,\; y\in D_0.$$

We observe that elements of $G$ and $H$ permute the clopen sets $\{C_i\}_{i=0}^{n-1}$  and $\{D_i\}_{i = 0}^{n-1}$, respectively.  Fix two families of coset representatives $F_G$ and $F_H$ for $G/G_0$ and $H/H_0$, respectively. Assume that the identities of the respective groups belong to $F_G$ and $\in F_H$. Fix a bijection $F_H\ni f\mapsto a_f\in F_G$ with $e \mapsto e$.

Extend the homeomorphism $\varphi$ to $\varphi: Y\rightarrow X$ by setting $$\varphi|_{f\cdot D_0}(y) = a_f \cdot  \varphi(f^{-1}\cdot y) $$ for every $f\in F_H$ and $y\in f\cdot D_0$.  We claim that $\varphi$ implements a continuous orbit equivalence between the systems.

Let $y\in D_0$ and $h\in H$. Write $h = fh_0$, where $f\in F_H$ and $h_0\in H_0$. Then
\begin{equation}\label{EqClopenNbhdOE}\varphi(h\cdot y ) = \varphi(fh_0\cdot y) = a_f\cdot \varphi(h_0\cdot y) = a_f \theta(h_0)\cdot \varphi(y).\end{equation}
 Thus, $\varphi(y)$ and $\varphi(h\cdot y)$ lie in the same $G$-orbit. Therefore, $\varphi (H\cdot y) \subset G\cdot \varphi(y)$ for every $y\in D_0$. Since the set  $D_0$ meets every $H$-orbit,  the inclusion extends to any $y\in Y$. Using the same argument as above one can also show that $\varphi (H\cdot y) \supset G\cdot \varphi(y)$ for every $y\in Y$.
 Thus, $\varphi$ implements an orbit equivalence between the systems $(X,G)$ and $(Y,H)$. Since Equation (\ref{EqClopenNbhdOE}) holds on a clopen neighborhood of $y$, we conclude that $\varphi$ implements a continuous orbit equivalence.
\end{proof}

In the following result we show that for abelian groups continuous orbit equivalence implies virtual piecewise conjugacy.

\begin{theorem}\label{TheoremRigidityZnOdometers}  Let $(X,G)$ and $(Y,H)$ be  free odometers, where  $G = H =\mathbb Z^d$.  Let $\varphi : Y\rightarrow X$ be a map implementing a continuous orbit equivalence between  $(X,G)$ and $(Y,H)$ and let $f: H\times Y \rightarrow G$ be the corresponding orbit cocycle.  Then there exist finite index normal subgroups $G_0\subset G$ and $H_0\subset H$ with  $[G:G_0] = [H:H_0]$,  a clopen partition $\{C_0,\ldots,C_{q-1}\}$ of $Y$ into $H_0$-invariant sets, and isomorphisms $\theta_i : H_0\rightarrow G_0$, $i = 0,\ldots, q-1$, such that  $f(h,y) = \theta_i(h)$ for every $h\in H_0$ and $y\in C_i$, $i=0,\ldots,q-1$.
\end{theorem}
\begin{proof}  Without loss of generality, we can assume that both groups act on $X$. Let $H'$ and $\{C_0,\ldots, C_{q-1}\}$ be as in the proof of Statement (1) in  Theorem \ref{TheoremMainRigidity}.
Since $H'$ is a normal subgroup, every set $C_i$ is $H'$-invariant. Note also that $f(h,x)=f(h,y)$ for every $h\in H$ and $x,y\in C_i$.

 Fix $i=0,\ldots,q-1$  and $x\in C_i$. Set $\theta_i(h) = f(h,x)$. Set $G_i' = \theta_i(H')$. As in the proof of Theorem \ref{TheoremMainRigidity}, we obtain that the definition of $\theta_i$ is independent of $x\in C_i$ and $\theta_i : H'\rightarrow G_i'$ is an isomorphism of the groups. Furthermore,  $[G:G_i'] = [H:H'] <\infty$ for every $i=0,\ldots,q-1$.

We claim that  $G_i=G_j$ for all $i$ and $j$.  Observe that, by construction of exact odometers, the subgroup $H'$ possesses the property that $H'=\{h\in H : h(x) \in C_i\}$ for every $i$ and $x\in C_i$.  Let $h\in H'$ and $q\in H$. Then
  \begin{equation*}\begin{array}{lll}
 f(qhq^{-1},x) & = & f(q,hq^{-1}\cdot x)f(h,q^{-1}\cdot x) f(q^{-1},x) \\
 & = & f(q,q^{-1}\cdot x)f(h,q^{-1}\cdot x) f(q^{-1},x) \\
 & = & f(q^{-1},x)^{-1} f(h,q^{-1}\cdot x) f(q^{-1},x).
 \end{array}\end{equation*}

Therefore, if $x\in C_i$ and $q^{-1}\cdot x\in C_j$, $q\in H$, then $$G_i' = f(H',x) = f(q H' q^{-1},x) = f(q^{-1},x)^{-1}f(H',q^{-1}\cdot x)f(q^{-1},x) = G_j'.$$

  Set $G' = G_i'$. Thus, for every $i=0,\ldots,q-1$, we have that $\theta_i : H'\rightarrow G'$ is a group isomorphism  and $f(h,x) = \theta_i(h)$, for every $x\in C_i$ and $h\in H'$. In particular, this implies that  the systems $(C_i,H')$ and $(C_i,G')$ are topologically $\theta_i$-conjugate. This completes the proof.
\end{proof}

 The following result is an immediate consequence of Theorem \ref{TheoremMainRigidity} and  \cite[Theorem 2.3]{BoyleTomiyama:1998}. Recall that two $\mathbb Z$-actions $(X,T)$ and $(Y,T)$ are called {\it flip conjugate} if $(X,T)$ is conjugate to $(Y,S)$ or to $(Y,S^{-1})$.

\begin{corollary}\label{CorollaryVirtualFlipEquivalence} Let $(X,T)$ and $(Y,S)$ are $\mathbb Z$-odometers. Then $(X,T)$ and $(Y,S)$ are flip conjugate if and only if they are structurally conjugate.
\end{corollary}

 As an alternative proof of Corollary \ref{CorollaryVirtualFlipEquivalence}, one could also use the Gottschalk-Hedlung theorem and ideas from the proof of Theorems \ref{TheoremMainRigidity} and \ref{TheoremRigidityZnOdometers} to show that for $\mathbb Z$-odometers continuous orbit cocycles  are cohomologous to automorphisms of $\mathbb Z$, which would imply the result.

Given a dynamical system $(X,G)$, denote by $C_r^*(X,G)$ the reduced crossed product $C^*$-algebra arising from a dynamical system $(X,G)$, see, for example, \cite{Renault:2008}. The reduced crossed  product $C^*$-algebras are isomorphic via a map preserving $C(X)$ if and only if the associated dynamical systems are continuously orbit equivalent \cite[Theorem 5.1]{Matui:2012}, \cite[Proposition 4.13]{Renault:2008}. Applying this fact along with Theorem \ref{TheoremMainRigidity}, we obtain the following result classifying the restricted isomorphism class of generalized Bunce-Deddens algebras \cite{Orfanos:2010}.

\begin{corollary}\label{CorollaryBunceDeddensAlgebras} Let $(X,G)$ and $(Y,H)$ be  free odometers with $G$ and $H$ finitely generated residually finite groups. The following are equivalent:

(1) $(X,G)$ and $(Y,H)$ are structurally conjugate.

(2) There exists an isomorphism $\pi: C^*_r(X,G) \rightarrow C_r^*(Y,H)$ such that $\pi(C(X)) = C(Y)$.
\end{corollary}

We finish the section by presenting two structurally conjugate, but not conjugate  $\mathbb Z^2$-odometers. We note that similar examples were independently constructed in  \cite{Li:2015}.

\begin{example} Let $A_0$ and $B_0$ be the $2\times 2$ matrices given by
 $$
A_0=\left( \begin{array}{cc}
                              4 & 1\\
                               0   & 1
                               \end{array}
                                \right)
                                 \mbox{ and } B_0= \left(\begin{array}{cc}
                                                                      2&0\\
                                                                      0& 2\\
                                                                      \end{array}\right).
                                                                      $$
Set $$A_n = A_0^n\mbox{ and }B_n = B_0A_0^{n-1}.$$ Let $G = H = \mathbb Z^2$ and $G_n = A_n\mathbb Z^2$ and $H_n = B_n\mathbb Z^2$.  Consider the odometers $(X,G)$ and $(Y,H)$  determined by the sequences of normal subgroups $\{G_n\}_{n\geq 0}$ and $\{H_n\}_{n\geq 0}$, respectively. Note that $[G:G_0] = [H:H_0]$. Set $\theta = A_0 B_0^{-1}$. Then $\theta: G_0\rightarrow H_0$ is a group isomorphism. Furthermore, $\theta(G_n) = H_n$ for every $n\geq 1$. Thus, the odometers $(X,G)$ and $(Y,H)$ are structurally conjugate.

Assume that the dynamical systems $(X,G)$ and $(Y,H)$ are topologically conjugate. This means that  there  exists a group isomorphism (viewed as a matrix) $\Lambda:\ZZ^2\to \ZZ^2$  and a homeomorphism $\alpha: X \rightarrow Y$ such that $\alpha(g\cdot x) = \Lambda(g)\cdot \alpha(x)$ for every $g\in G$ and $x\in X$. In other words, the sequences of subgroups $\{\Lambda(G_n)\}_{n\geq 0}$ and $\{H_n\}_{n\geq 0}$ define the same, up to a space homeomorphism, odometers. These odometers are uniquely determined by the sequences of matrices $\{\Lambda A_n \}_{n\geq 0}$ and $\{B_n\}_{n\geq 0}$.
By \cite[Lemma 2]{Cor06} there exist a matrix $P\in GL_2(\mathbb Z)$ and $n\geq 0$ such that $B_0P=\Lambda A_n$.  Hence $2 \Lambda^{-1}P = A_0^{n}$. It follows that the entries of $A_0^{n}$ are divisible by 2, which is a contradiction. Therefore, $(X,G)$ and $(Y,H)$ cannot be topologically conjugate.
\end{example}

%
%
%

%
%

\section{Topological Full Groups}\label{SectionTopFullGroups}

This section is devoted to the study of topological full groups of $G$-odometers. We will show that topological full groups ``know'' when the underlying dynamical systems are odometers. We also show that  the topological full group of a $G$-odometer is amenable if and only if $G$ is amenable. In particular, this implies that the topological full group of a product of $\mathbb Z$-odometers is amenable since such systems can be obtained as $\mathbb Z^d$-odometers using the diagonal scale matrices (see \cite{Cor06} for more details).

Let $(X,G)$ be a  Cantor minimal  system. Denote by $[[G]]$ the group of homeomorphisms $s : X\rightarrow X$ such that for every $x\in X$ there exists a clopen neighbourhood $U$ of $x$ and an element $g\in G$ such that $s(y) = g(y)$ for every $y\in U$.

\begin{definition} The group $[[G]]$ is called the {\it topological full group of}  $(X,G)$.
\end{definition}

We refer the reader to the paper \cite{GrigorchukMedynets} surveying algebraic properties of full groups. The following result shows that the topological full group $[[G]]$ is a complete invariant of continuous orbit equivalence \cite[Remark 2.11]{Med11}. We note that the original result \cite{Med11} was established under much weaker assumptions than those presented here.

\begin{theorem}\label{TheoremIsomorphismFullGroups} Let $(X,G)$ and $(Y,H)$ be Cantor minimal systems. Then $(X,G)$ and $(Y,H)$ are continuously orbit equivalent if and only if the topological full groups $[[G]]$ and $[[H]]$ are isomorphic as abstract groups.

Furthermore, for every group isomorphism $\alpha :[[G]] \rightarrow [[H]]$ there exists a homeomorphism $\Lambda : X\rightarrow Y$ such that $\alpha(g) = \Lambda \circ g \circ \Lambda^{-1}$ for all $g\in [[G]]$.
\end{theorem}

Combining Theorems \ref{TheoremIsomorphismFullGroups} and \ref{TheoremMainRigidity}, we immediately obtain the following results, which, in particular, show that topological full groups ``know'' when the underlying systems are equicontinuous.

\begin{corollary}\label{CorollaryFullGroupsVirtuallyConjugate} Let $(X,G)$ and $(Y,H)$ be  free odometers with $G$ and $H$ finitely generated residually finite groups.  Then the
topological full groups $[[G]]$ and $[[H]]$ are isomorphic as abstract groups if and only if the systems $(X,G)$ and $(Y,H)$ are structurally  conjugate.
\end{corollary}

\begin{corollary} Let $(X,G)$ be a free minimal equicontinuous system. Let $(Y,H)$ be a free dynamical system. Suppose that the topological full groups $[[G]]$ and $[[H]]$ are isomorphic as abstract groups. Then $(Y,H)$ is minimal and equicontinuous.
\end{corollary}

 Consider a free exact  $G$-odometer $(X,G)$ determined by a sequence of finite-index subgroups $(G_i)_{i\geq 0}$.
 For every $n\geq 0$, let $C_n$ be the subset of all $x=(x_k)_{k\geq 0}\in X$ such that
 $x_n = e_n$. The collection $\P_n=\{f\cdot C_n: f\in F_n\}$ is a clopen partition of $X$, where $F_n$ is a  set of representatives of $G/G_n$. We note that the partition $\P_n$ is independent of the choice of $F_n$. The family of partitions $(\P_n)_{n\geq 0}$ spans the topology of $X$.  Note also that the elements of $G$ permutes the atoms of $\P_n$.

  Let $[[G]]$ be the topological full group of $(X,G)$.   For each $\gamma \in [[G]]$ and $x\in X$, let  $f(\gamma,x)\in G$ be such that $f(\gamma,x)\cdot x= \gamma \cdot x$. Since the group $G$ acts freely, the function $f : [[G]]\times X\rightarrow G$ is well-defined. Note that $f$ satisfies the cocycle identity: \begin{equation}\label{EqCocycleIdentity}f(\gamma_1\gamma_2,x)  = f(\gamma_1,\gamma_2\cdot x) f(\gamma_2,x)\mbox{ for every }\gamma_1,\gamma_2\in G\end{equation}

For every $n\geq 0$, denote by
  $[[G]]_n$ the set of all $\gamma\in [[G]]$ such that the cocycle $f(\gamma,\cdot)$ is $\P_n$-compatible, i.e., constant on atoms of $\P_n$. Using the cocycle identity (\ref{EqCocycleIdentity}) and the fact that the group $G$ permutes the atoms of $\P_n$, we see that $[[G]]_n$ is a group.  Setting $\P_0=\{X\}$, we have that $[[G]]_0=G$. The following result follows from the fact that  $(\P_n)_{n\geq 0}$ is a nested sequence of partitions spanning the topology of $X$.

\begin{proposition}\label{PropositionFullGroupsIncreasingUnion}$[[G]] = \bigcup_{n\geq 0}[[G]]_n$ and $[[G]]_n\subset [[G]]_{n+1}$ for every $n\geq 0$.
\end{proposition}

Denote by $S_p$ the symmetric group on $p$ elements.

 \begin{proposition}\label{PropositionOdometer-direct-limit} Suppose  that $(X,G)$ is a free  exact free odometer defined by a sequence of normal subgroups $\{G_n\}_{n=0}^\infty$.

 (1) The group $[[G]]_n$ is isomorphic to the semidirect product $G_n^{[G:G_n]}\rtimes S_{[G:G_n]}$.

(2) The topological full group $[[G]]$  is isomorphic to the inductive limit $$\varinjlim(G_n^{[G:G_n]}\rtimes S_{[G:G_n]},\tau_n).$$
\end{proposition}
\begin{proof}
Let $\{C_n\}$  and $\{\P_n\}$ be as above.  Recall that the group $G$ permutes atoms of $\P_n$. Consider $s\in [[G]]_n$. Then the orbit cocycle $f(s,\cdot)$ is constant on every atom of $\P_n$. Fix a family of representatives $F_n$ defining the partition $\mathcal P_n$.

For  $g\in F_n$ and $x\in C_n$, denote  $f(s,g\cdot x)$ by $u_g\in G$. Then $$s\cdot (g\cdot x) = f(s,gx)g\cdot x = u_gg\cdot x = g'q\cdot x,$$ where $g'\in F_n$ and $q\in G_n$. This shows that  $s$ permutes atoms of $\P_n$. Denote by $i_n : [[G]]_n\rightarrow S_{\P_n}$ the induced homomorphism. Note that $S_{\P_n} \cong S_{[G:G_n]}$.

The kernel of $i_n$ is the subgroup of $[[G]]_n$ that stabilizes every atom of $\P_n$. Since the groups $\{G_n\}$ are normal, $G_n$ is the set of return times to $f\cdot C_n$ for any  $f\in F_n$. Thus, for every $s\in \mathrm{Ker}(\psi)$ and every $q\cdot C_n$, $q\in F_n$, we have that $s|_{q(C_n)} = q'|_{q(C_n)}$ for some $q'\in G_n$. It follows that
$$\mathrm{Ker}(\psi) \cong G_n^{[G:G_n]}.$$
Therefore, $[[G]]_n$ is isomorphic to the semidirect product $G_n^{[G:G_n]}\rtimes S_{[G:G_n]}$. Applying Proposition \ref{PropositionFullGroupsIncreasingUnion} we complete the proof.
\end{proof}

The following result is a corollary of Proposition \ref{PropositionOdometer-direct-limit}. We recall that semidirect products of amenable groups and inductive limits of amenable groups are amenable. For more details on properties of amenable groups, see, for example, \cite[Section 4.5]{CC10}.

\begin{corollary}\label{amenable}
The topological full group $[[G]]$ of a free exact odometer $(X,G)$  is amenable if and only if $G$ is amenable.
\end{corollary}

We will need the following folklore result.
\begin{lemma}\label{LemmaEmbeddingFullGroups} Let $(X,G)$ be a Cantor system. Suppose that $(Y,G)$ is a factor of $(X,G)$. Then the topological full group of $(Y,G)$ embeds into the topological full group of $(X,G)$.
\end{lemma}
\begin{proof} Let $\pi : (X,G) \rightarrow (Y,G)$ be a factor map. For $s\in [[(Y,G)]]$, find a clopen partition $Y = \bigsqcup Y_i$ and elements $\{g_i\}\subset G$ such that $s|_{A_i} = g_i|_{A_i}$. Define $\bar \pi(s)$ as an element of $[[(X,G)]]$ such that $$\bar \pi(s)|_{\pi^{-1}(Y_i)} = g_i|_{{\pi^{-1}(Y_i)}}.$$ Since $\pi$ is $G$-equivariant, $\pi(s)$ is well-defined.  Now it is routine to check that $\bar \pi$ is a group embedding.
\end{proof}

\begin{corollary}\label{CorollaryFullGroupsAmenable}
Let $(X,G)$ an  equicontinuous free minimal Cantor system. Then  $[[G]]$ is amenable if and only if $G$ is amenable.
\end{corollary}
\begin{proof} By Theorem \ref{TheoremEquicontinuousConjugateOdometers}, $(X,G)$ is conjugate to a free odometer.
By \cite[Proposition 1]{CP}, the dynamical system $(X,G)$ is a factor of an exact odometer.

Suppose that $G$ is amenable.  By Lemma \ref{LemmaEmbeddingFullGroups},
$[[G]]$ is isomorphic to a subgroup of the topological full group of an exact odometer. Hence,  Corollary \ref{amenable} implies that $[[G]]$ is amenable.

Conversely, if  $[[G]]$ is amenable, then $G$ is amenable as a subgroup of  $[[G]]$.
\end{proof}

%
%

%
%
%

%
%

\end{document}